\documentclass[12pt,twoside]{article}

\usepackage{amsmath, amsthm}

\usepackage{amssymb}
\usepackage{epsfig}
\usepackage{mathdots}
 \theoremstyle{plain}
\newtheorem{theorem}{Theorem}
\newtheorem{corollary}{Corollary}

\newtheorem{lemma}{Lemma}
\newtheorem{proposition}{Proposition}
\newtheorem{example}{Example}
\theoremstyle{definition}
\newtheorem{definition}{Definition}
\theoremstyle{remark}

\numberwithin{equation}{section}

\newcommand{\bT}{\begin{theorem}}
\newcommand{\eT}{\end{theorem}}
\newcommand{\bProp}{\begin{proposition}}
\newcommand{\eProp}{\end{proposition}}
\newcommand{\bE}{\begin{example}}
\newcommand{\eE}{\end{example}}
\newcommand{\bL}{\begin{lemma}}
\newcommand{\eL}{\end{lemma}}
\newcommand{\bP}{\begin{proof}}
\newcommand{\eP}{\end{proof}}
\newcommand{\bC}{\begin{corollary}}
\newcommand{\eC}{\end{corollary}}
\newcommand{\bD}{\begin{definition}}
\newcommand{\eD}{\end{definition}}
\newcommand{\be}{\begin{enumerate}}
\newcommand{\ee}{\end{enumerate}}
\newcommand{\beqa}{\begin{eqnarray*}}
\newcommand{\eeqa}{\end{eqnarray*}}
\newcommand{\beqaa}{\begin{eqnarray}}
\newcommand{\eeqaa}{\end{eqnarray}}
\newcommand{\ba}{\begin{array}}
\newcommand{\ea}{\end{array}}

\setbox0=\hbox{$+$}
\newdimen\plusheight
\plusheight=\ht0
\def\+{\;\lower\plusheight\hbox{$+$}\;}

\setbox0=\hbox{$-$}
\newdimen\minusheight
\minusheight=\ht0
\def\-{\;\lower\minusheight\hbox{$-$}\;}

\setbox0=\hbox{$\cdots$}
\newdimen\cdotsheight
\cdotsheight=\plusheight
\def\cds{\lower\cdotsheight\hbox{$\cdots$}}
\begin{document}




\centerline {\Large{\bf Some Applications of a Bailey-type
Transformation}}

\centerline{}


\centerline{}

\centerline{\bf {James Mc Laughlin}}

\centerline{}

\centerline{Mathematics Department}

\centerline{25 University Ave.,}

\centerline{West Chester University}

\centerline{West Chester, PA 19383, USA}

\centerline{}

\centerline{\bf {Peter Zimmer}}

\centerline{}

\centerline{Mathematics Department}

\centerline{25 University Ave.,}

\centerline{West Chester University}

\centerline{West Chester, PA 19383, USA}

\newtheorem{Theorem}{\quad Theorem}[section]

\newtheorem{Definition}[Theorem]{\quad Definition}

\newtheorem{Corollary}[Theorem]{\quad Corollary}

\newtheorem{Lemma}[Theorem]{\quad Lemma}

\newtheorem{Example}[Theorem]{\quad Example}

\date{\today}

\begin{abstract}
If  $k$ is set equal to $a q$ in the definition of a WP Bailey pair,
\[
\beta_{n}(a,k) = \sum_{j=0}^{n}
\frac{(k/a)_{n-j}(k)_{n+j}}{(q)_{n-j}(aq)_{n+j}}\alpha_{j}(a,k),
\]
this equation  reduces to $\beta_{n}=\sum_{j=0}^{n}\alpha_{j}$.

This seemingly trivial relation connecting the $\alpha_n$'s with the
$\beta_n$'s has some interesting consequences, including several
basic hypergeometric summation formulae, a connection to the
Prouhet-Tarry-Escott problem, some new identities of the
Rogers-Ramanujan-Slater type, some new expressions for false theta
series as basic hypergeometric series, and new transformation
formulae for poly-basic hypergeometric series.

\end{abstract}




\section{Introduction}

In the present paper we show that the simple relationship
$\beta_{n}=\sum_{j=0}^{n}\alpha_{j}$ connecting two sequences $\{\alpha_n\}$ and $\{\beta_n\}$ has some interesting and surprising consequences. These include several new
basic hypergeometric transformation- and summation formulae, a connection to the
Prouhet-Tarry-Escott problem, some new identities of the
Rogers-Ramanujan-Slater type, some new expressions for false theta
series as basic hypergeometric series, and new transformation
formulae for poly-basic hypergeometric series. We give some examples now to illustrate these (details may be found later in the paper). The following identity is one of several new summation formulae.
\begin{equation*}
\sum_{n=0}^{\infty}
\frac{(1+q^{n+1}/x)(q/x^2;q)_{n}x^{n}(1-u^{n+1})}{(q;q)_{n+1}} =
\frac{1-u} {1-x}\frac{(qu/x;q)_{\infty}}{(xu;q)_{\infty}}.
\end{equation*}

We also found the following new transformation for poly-basic series.
\begin{multline*}
\sum_{n=0}^{\infty}
\frac{(q\sqrt{xyz},-q\sqrt{xyz}, y,
z;q)_{n}}{(\sqrt{xyz},-\sqrt{xyz},qxy,qxz;q)_n} \frac{ (ap^2;p^2)_n
\left(b P^2;P^2 \right)_n  } {
\left(\frac{PQR}{p};\frac{PQR}{p}\right)_n \left(\frac{a p P Q}{c
R};\frac{p P Q}{R}\right)_n}\\
\times
 \frac{ (cR^2;R^2)_n
\left(\frac{a Q^2}{b c};Q^2 \right)_n  } { \left(\frac{a
pQR}{bP};\frac{pQR}{P}\right)_n \left(\frac{b c p P R}{Q};\frac{ p P
R}{Q}\right)_n}\,x^{n}
= \frac{\left(1-xy\right)\left(1-xz\right)}
{(1-x)\left(1-xyz\right)} \times \\
\sum_{n=0}^{\infty} \frac{( y, z;q)_{n}}{(xy,xz;q)_n} \frac{
\left(1-a p^n P^n Q^n R^n\right)\left(1-b \frac{p^n P^n}{ Q^{n}
R^{n}}\right) \left(1-\frac{ P^n Q^n
  }{c p^{n} R^{n}}\right) \left(1-\frac{a p^n
   Q^n }{b cP^{n}R^{n}}\right)}{(1-a) (1-b)
   \left(1-\frac{1}{c}\right) \left(1-\frac{a}{b c}\right)}
\\
\times
 \frac{ (a;p^2)_n
\left(b ;P^2 \right)_n  } {
\left(\frac{PQR}{p};\frac{PQR}{p}\right)_n \left(\frac{a p P Q}{c
R};\frac{p P Q}{R}\right)_n}
 \frac{ (c;R^2)_n
\left(\frac{a }{b c};Q^2 \right)_n  } { \left(\frac{a
pQR}{bP};\frac{pQR}{P}\right)_n \left(\frac{b c p P R}{Q};\frac{ p P
R}{Q}\right)_n}\,(x R^2)^{n}.
\end{multline*}

If
\begin{equation*}
\{a_1, . . . , a_{12}\}\stackrel{11}{=}\{b_1, . . . , b_{11},1\}
\end{equation*}
is a solution to the \emph{Prouhet-Tarry-Escott problem} (see later for explanations and an explicit set of values for the $a_i$'s and $b_i$'s), then
\begin{multline*}
\,_{12} \phi _{11} \left [
\begin{matrix}
  a_1,\, a_2,\, a_{3},\, a_4,\,a_5,\,a_6,\,a_7,\,a_8,a_9,a_{10},a_{11},a_{12} \\
b_1q,\, b_2 q,\, b_{3}q,\,b_4 q,\,b_5 q,\,b_6 q,b_7
q,b_9q,b_{10}q,b_{11}q
\end{matrix}
; q,\, q^{12} \right ].\\
=\frac{(a_1q,a_2q,a_{3}q,a_4q,a_5q,a_6q,a_7q,a_8q,a_9q,a_{10}q,a_{11}q,a_{12}q;q)_{\infty}}
{(b_1q,\, b_2 q,\, b_{3}q,\,b_4 q,\,b_5 q,\,b_6 q,b_7
q,b_9q,b_{10}q,b_{11}q,q;q)_{\infty}}.
\end{multline*}
New identities of Rogers-Ramanujan-Slater type and new representations for false theta series include the following identities.
\begin{equation*}
1+\sum_{n=1}^{\infty} \frac{(-q;q)_n q^{(n^2-n)/2}}{(q;q)_{n-1}} =
\frac{(-1;q)_{\infty}(-q^6,-q^{10},q^{16};q^{16})_{\infty}}{(q^4;q^4)_{\infty}}.
\end{equation*}
\begin{equation*}
\frac{1}{2}+\sum_{n=0}^{\infty}\frac{(-1)^nq^{n(n+1)/2}}{(-1;q)_{n+2}}=\sum_{n=0}^{\infty}
q^{n(3n+1)/2}(1-q^{2n+1}).
\end{equation*}

\subsection{Background}
We begin by recalling  Andrews'  construction \cite{A01} of a
\emph{WP-Bailey chain}. If a pair of sequences
$(\alpha_{n}(a,k),\,\beta_{n}(a,k))$ satisfy
{\allowdisplaybreaks
\begin{equation}\label{WPpair}
\beta_{n}(a,k) = \sum_{j=0}^{n}
\frac{(k/a)_{n-j}(k)_{n+j}}{(q)_{n-j}(aq)_{n+j}}\alpha_{j}(a,k),
\end{equation}}
then so does the pair $(\alpha_{n}'(a,k),\,\beta_{n}'(a,k))$, where
{\allowdisplaybreaks
\begin{align}\label{wpn1}
\alpha_{n}'(a,k)&=\frac{(\rho_1, \rho_2)_n}{(aq/\rho_1,
aq/\rho_2)_n}\left(\frac{k}{c}\right)^n\alpha_{n}(a,c),\\
\beta_{n}'(a,k)&=\frac{(k\rho_1/a,k\rho_2/a)_n}{(aq/\rho_1,
aq/\rho_2)_n} \notag\\
&\phantom{as}\times \sum_{j=0}^{n} \frac{(1-c
q^{2j})(\rho_1,\rho_2)_j(k/c)_{n-j}(k)_{n+j}}{(1-c)(k\rho_1/a,k\rho_2/a)_n(q)_{n-j}(qc)_{n+j}}
\left(\frac{k}{c}\right)^j\beta_{j}(a,c), \notag
\end{align}
}with  $c=k\rho_1 \rho_2/aq$.  A pair of sequences satisfying
\eqref{WPpair} is termed a \emph{WP-Bailey pair}.  The process may
be iterated to produce a sequence of WP-Bailey pairs from an initial
pair, hence the designation ``WP-Bailey chain" for this process.
Andrews also described a second WP-Bailey chain in \cite{A01}, and
these two constructions allow a ``tree" of WP-Bailey pairs to be
generated from a single WP-Bailey pair.

The implications of these two branches were further investigated  in
\cite{AB02} by Andrews and Berkovich. An elliptic generalization of
Andrews first WP-Bailey chain  was derived in \cite{S02} by
Spiridonov. Warnaar \cite{W03} added four additional branches to the
WP-Bailey tree, two of which had generalizations to the elliptic
level. More recently, Liu and Ma \cite{LM08} introduced the idea of
a general WP-Bailey chain, and added one new branch to the WP-Bailey
tree. In \cite{MZ09}, the authors added three new WP-Bailey chains.

If $k=0$, the pair of sequences become what is termed a \emph{Bailey
pair relative to $a$}. Bailey \cite{B47, B49}  used the $q$-Gauss
sum,
\begin{equation}\label{qgauss}
_2\phi_1 (a,b;c;q,c/ab)=
\frac{(c/a,c/b;q)_{\infty}}{(c,c/ab;q)_{\infty}},
\end{equation}to get that, if $(\alpha_n, \beta_n)$ is a
Bailey pair relative to $a$, then
\begin{equation}\label{Baileyeq}
\sum_{n=0}^{\infty}(y,z;q)_{n}\left(\frac{aq}{yz}\right)^{n}\beta_n
=\frac{\left(\frac{aq}{y},\frac{aq}{z};q\right)_{\infty}}{\left(aq,
\frac{aq}{yz};q\right)_{\infty}}\sum_{n=0}^{\infty}
\frac{(y,z;q)_{n}}{\left(\frac{aq}{y},\frac{aq}{z};q\right)_n}\left
(\frac{x}{yz}\right)^{n}\alpha_n.
\end{equation}

As might be expected, Andrews generalization  of a Bailey pair leads
to a generalization of \eqref{Baileyeq}. Indeed Andrews WP-Bailey
chain at \eqref{wpn1}  can easily be shown to imply the following
result (substitute the expression for $\alpha_{n}'(a,k)$ in
\eqref{WPpair}, set the two expressions for $\beta_{n}'(a,k)$ equal,
and let $n \to \infty$). Note that setting $k=0$ below recovers
Bailey's transformation at \eqref{Baileyeq}.

\begin{theorem}\label{6t1}
Under suitable convergence conditions, if
$(\alpha_n(a,k),\beta_n(a,k))$ satisfy \eqref{WPpair}, then
\begin{multline}\label{6betaneq}
\sum_{n=0}^{\infty} \frac{(1-kq^{2n})(\rho_1,\rho_2;q)_n }{(1-k)(k
q/\rho_1,k q/\rho_2;q)_n}\,\left(\frac{a q}{\rho_1 \rho_2}\right)^n
\beta_n(a,k)=\\\frac{(k q,k q/\rho_1\rho_2,a q/\rho_1,a
q/\rho_2;q)_{\infty}}{(k q/\rho_1,k q/\rho_2,a q/\rho_1 \rho_2,a
q;q)_{\infty}}
\sum_{n=0}^{\infty} \frac{(\rho_1,\rho_2;q)_{n}} {(a q/\rho_1,a
q/\rho_2;q)_{n}}\left ( \frac{a q}{\rho_1 \rho_2}\right)^n
\alpha_n(a,k).
\end{multline}
\end{theorem}

In the present paper we investigate what at first glance  may appear
to be a trivial  special case of Theorem \ref{6t1}.

\begin{corollary}\label{6c1}
If $\beta_n = \sum_{r=0}^{n} \alpha_{r}$,  then assuming both series
converge,
\begin{equation}\label{6simsum}
\sum_{n=0}^{\infty} \frac{(q\sqrt{xyz},-q\sqrt{xyz}, y, z;q)_{n}x^{n} \beta_n}{(\sqrt{xyz},-\sqrt{xyz},qxy,qxz;q)_n} \\
= \frac{\left(1-xy\right)\left(1-xz\right)}
{(1-x)\left(1-xyz\right)} \sum_{n=0}^{\infty} \frac{(y,z;q)_{n}x^{n}
\alpha_n}{(xy,xz;q)_n}.
\end{equation}
\end{corollary}
\begin{proof}
Let $k=xyz$, $a=xyz/q$, $\rho_1=y$ and $\rho_2=z$  in Theorem
\ref{6t1}.
\end{proof}

This seemingly trivial relation connecting the $\alpha_n$'s with the
$\beta_n$'s (of course the pairs of sequences $(\alpha_n,\beta_n)$
defined by this relation are no longer Bailey pairs or WP-Bailey
pairs)  has some interesting consequences, including several new
basic hypergeometric summation formulae, a connection to the
Prouhet-Tarry-Escott problem, some new identities of the
Rogers-Ramanujan-Slater type, some new expressions for false theta
series as basic hypergeometric series, and new transformation
formulae for poly-basic hypergeometric series.


As usual, for $a$ and $q$ complex numbers with $|q|<1$, define
\begin{align*}
&(a)_0 =(a;q)_0 :=1, \hspace{20pt} (a)_n=(a;q)_n
:=\prod_{j=0}^{n-1}(1-a q^j), \text{ for } n\in \mathbb{N},\\
&(a_1;q)_n(a_2;q)_n \dots (a_k;q)_n = (a_1,a_2,\dots, a_k;q)_n,\\
&(a;q)_{\infty}:=\prod_{j=0}^{\infty}(1-a q^j), \hspace{20pt}\\
&(a_1;q)_{\infty}(a_2;q)_{\infty} \dots (a_k;q)_{\infty} =
(a_1,a_2,\dots, a_k;q)_{\infty}.
\end{align*}
    An $_{r} \phi _{s}$ basic hypergeometric series is defined by
\begin{equation*} _{r} \phi _{s} \left [
\begin{matrix}
a_{1},\dots, a_{r}\\
b_{1}, \dots, b_{s}
\end{matrix}
; q,x \right ]=  \sum_{n=0}^{\infty} \frac{(a_{1}\dots a_{r};q)_{n}}
{(q,b_{1}\dots b_{s};q)_{n}} \left( (-1)^{n} q^{n(n-1)/2} \right
)^{s+1-r}x^{n}.
\end{equation*}
 For future use we also recall the $q$-binomial theorem,
\begin{equation}
\label{qbinom} \sum_{n=0}^{\infty}\frac{(a;q)_n}{(q;q)_n}z^n =
\frac{(az;q)_{\infty}}{(z;q)_{\infty}}.
\end{equation}

\section{Various Transformation- and Summation\\ Formulae for Basic Hypergeometric series}

We next derive a number of transformation formulae for basic
hypergeometric series, transformations that give rise  to summation
formulae for particular choices of the parameters. We believe these
to be new.

\begin{corollary}\label{6c1aa}
For $q$ and $x$  inside the unit disc,
\begin{multline}\label{6aln=1a}
\sum_{n=0}^{\infty} \frac{(q\sqrt{xyz},-q\sqrt{xyz}, y,
z;q)_{2n}x^{2n} }
{(\sqrt{xyz},-\sqrt{xyz},qxy,qxz;q)_{2n}} \\
= \frac{\left(1-xy\right)\left(1-xz\right)}
{(1-x)\left(1-xyz\right)} \sum_{n=0}^{\infty}
\frac{(y,z;q)_{n}(-x)^{n}}{(xy,xz;q)_n}.
\end{multline}
\begin{equation}\label{c61aeq2}
\sum_{n=0}^{\infty}\frac{(1-q^{2n+1}/x)(q/x^2;q)_{2n}
x^{2n}}{(q;q)_{2n+1}}
=\frac{1}{1+x}\frac{(q/x;q)_{\infty}}{(x;q)_{\infty}}, \,\,\,x\not
=0.
\end{equation}
\end{corollary}
\begin{proof}
In Corollary \ref{6c1} let $\alpha_r=(-1)^r$ to get \eqref{6aln=1a}.
Then set $y=q/x$, $z=q/x^2$, apply \eqref{qbinom} to the right side,
replace $x$ by $-x$ and \eqref{c61aeq2} follows.
\end{proof}

\begin{corollary}\label{6c1a}
For $q$ and $x$  inside the unit disc,
\begin{multline}\label{6aln=1}
\sum_{n=0}^{\infty} \frac{(q\sqrt{xyz},-q\sqrt{xyz}, y, z;q)_{n}x^{n} (n+1)}
{(\sqrt{xyz},-\sqrt{xyz},qxy,qxz;q)_n} \\
= \frac{\left(1-xy\right)\left(1-xz\right)}
{(1-x)\left(1-xyz\right)} \sum_{n=0}^{\infty}
\frac{(y,z;q)_{n}x^{n}}{(xy,xz;q)_n}.
\end{multline}
\begin{equation}\label{c61eq2}
\sum_{n=0}^{\infty}\frac{(1+q^{n+1}/x)(q/x^2;q)_nx^n(n+1)}{(q;q)_{n+1}}
=\frac{1}{1-x}\frac{(q/x;q)_{\infty}}{(x;q)_{\infty}}.
\end{equation}
\end{corollary}
\begin{proof}
Set $\alpha_n=1$ in Corollary \ref{6c1} to get \eqref{6aln=1}. The
identity at \eqref{c61eq2} follows from \eqref{6aln=1} upon setting
$y=q/x^2$, $z=q/x$, using \eqref{qbinom} to sum the right side and
then simplifying.
\end{proof}


\begin{corollary}\label{6c2}
For $q$, $x$ and $u$ all inside the unit disc,
\begin{equation}\label{6qbineq}
\sum_{n=0}^{\infty}
\frac{(1+q^{n+1}/x)(q/x^2;q)_{n}x^{n}(1-u^{n+1})}{(q;q)_{n+1}} =
\frac{1-u} {1-x}\frac{(qu/x;q)_{\infty}}{(xu;q)_{\infty}}.
\end{equation}
\end{corollary}
\begin{proof}
Set $\alpha_n=u^n$, $y=q/x$ and $z=q/x^2$. Now apply the
$q$-binomial theorem \eqref{qbinom} to the right side.
\end{proof}
Remark: The identity at \eqref{c61eq2} above also follows from \eqref{6qbineq}, upon dividing through by $1-u$ and letting $u \to 1$, but we prefer to include both derivations.

\begin{corollary}\label{6c6}
\begin{equation}\label{5phi41}
_{5} \phi _{4} \left [\begin{matrix}
q\sqrt{xyz},-q\sqrt{xyz}, y,z, cq\\
\sqrt{xyz},-\sqrt{xyz}, qxy,qxz
\end{matrix}
; q,x \right ]\\
= \frac{(1-xy)(1-xz)}{ (1-x)(1-xyz)} \;_{3}\phi_{2} \left
[\begin{matrix}
y, z, c\\
xy,xz
\end{matrix}
; q,x q \right ].
\end{equation}
\begin{equation}\label{63phi22}
_{3} \phi _{2} \left [\begin{matrix}
-qxy, y, x\\
-xy, qx^2y
\end{matrix}
; q,x \right ]\\
= \frac{1}{1+xy}\frac{(x^2,qxy;q)_\infty}{(qx^2y,x;q)_\infty}.
\end{equation}
\end{corollary}
\begin{proof}
We define $\alpha_0 =1$, and for $n>0$,
\begin{equation*}
 \alpha_n =
 \frac{(cq;q)_n}{(q;q)_n}-\frac{(cq;q)_{n-1}}{(q;q)_{n-1}}
 =\frac{(c;q)_n  }{(q;q)_n}q^{n}.
 \end{equation*}
 Substitution into  \eqref{6simsum} immediately gives \eqref{5phi41}.
Equation \eqref{63phi22} follows upon letting $c=x/q$, $z=xy$ and
using \eqref{qgauss} to sum the resulting right side and
simplifying.
\end{proof}


\section{Transformation Formulae for basic- and \\polybasic
Hypergeometric Series} Most (possibly all) summation formulae for
poly-basic hypergeometric series arise because the series involved
telescope. This means that the terms in such an identity may be
inserted in \eqref{6simsum} to produce a transformation formula for
polybasic hypergeometric series containing an additional base.
Setting all the bases equal to $q^m$, for some integer $m$, then
gives a transformation formula for basic hypergeometric series. We
give one example in the next corollary, which contains a
transformation formula connecting polybasic hypergeometric series
with five independent bases.

\begin{corollary}\label{pb2}
Let $P$, $p$, $Q$, $q$, $R$ and $x$ all lie inside the unit disc,
and let $a$, $b$, $c$,  $y$ and $z$ be complex numbers such that
 the denominators below are bounded away from zero. Then {\allowdisplaybreaks
\begin{multline}\label{poly2} \sum_{n=0}^{\infty}
\frac{(q\sqrt{xyz},-q\sqrt{xyz}, y,
z;q)_{n}}{(\sqrt{xyz},-\sqrt{xyz},qxy,qxz;q)_n} \frac{ (ap^2;p^2)_n
\left(b P^2;P^2 \right)_n  } {
\left(\frac{PQR}{p};\frac{PQR}{p}\right)_n \left(\frac{a p P Q}{c
R};\frac{p P Q}{R}\right)_n}\\
\times
 \frac{ (cR^2;R^2)_n
\left(\frac{a Q^2}{b c};Q^2 \right)_n  } { \left(\frac{a
pQR}{bP};\frac{pQR}{P}\right)_n \left(\frac{b c p P R}{Q};\frac{ p P
R}{Q}\right)_n}\,x^{n}
\\
= \frac{\left(1-xy\right)\left(1-xz\right)}
{(1-x)\left(1-xyz\right)} \times \\
\sum_{n=0}^{\infty} \frac{( y, z;q)_{n}}{(xy,xz;q)_n} \frac{
\left(1-a p^n P^n Q^n R^n\right)\left(1-b \frac{p^n P^n}{ Q^{n}
R^{n}}\right) \left(1-\frac{ P^n Q^n
  }{c p^{n} R^{n}}\right) \left(1-\frac{a p^n
   Q^n }{b cP^{n}R^{n}}\right)}{(1-a) (1-b)
   \left(1-\frac{1}{c}\right) \left(1-\frac{a}{b c}\right)}
\\
\times
 \frac{ (a;p^2)_n
\left(b ;P^2 \right)_n  } {
\left(\frac{PQR}{p};\frac{PQR}{p}\right)_n \left(\frac{a p P Q}{c
R};\frac{p P Q}{R}\right)_n}
 \frac{ (c;R^2)_n
\left(\frac{a }{b c};Q^2 \right)_n  } { \left(\frac{a
pQR}{bP};\frac{pQR}{P}\right)_n \left(\frac{b c p P R}{Q};\frac{ p P
R}{Q}\right)_n}\,(x R^2)^{n};
\end{multline}
} {\allowdisplaybreaks\begin{multline}\label{poly2q}
\sum_{n=0}^{\infty} \frac{(q\sqrt{xyz},-q\sqrt{xyz}, y,
z;q)_{n}}{(\sqrt{xyz},-\sqrt{xyz},qxy,qxz;q)_n} \frac{
(aq^m,bq^m,cq^m,\frac{a q^m}{b c};q^m)_n } {
\left(\frac{a}{c}q^m,\frac{a}{b}q^m,b cq^m,q^m;q^m\right)_n
}\,x^{n}=
\\
\frac{\left(1-xy\right)\left(1-xz\right)}{(1-x)\left(1-xyz\right)}
\sum_{n=0}^{\infty}\frac{(y,z;q)_{n}}{(xy,xz;q)_n}\frac{
(q^m\sqrt{a},-q^m\sqrt{a},a,b,c,\frac{a}{bc};q^m)_n(xq^m)^{n}}{
\left(\sqrt{a},-\sqrt{a},\frac{a}{c}q^m,\frac{a}{b}q^m,b
cq^m,q^m;q^m\right)_n}.
\end{multline}}
\end{corollary}
\begin{proof}
We use the special case $m=0$, $d=1$ of the identity of Subbarao and
Verma labeled (2.2) in \cite{SV99}, namely,
{\allowdisplaybreaks\begin{multline} \sum_{k=0}^{n}  \frac{
\left(1-a p^k P^k Q^k R^k\right)\left(1-b \frac{p^k P^k}{ Q^{k}
R^{k}}\right) \left(1-\frac{ P^k Q^k
  }{c p^{k} R^{k}}\right) \left(1-\frac{a p^k
   Q^k }{b cP^{k}R^{k}}\right)}{(1-a) (1-b)
   \left(1-\frac{1}{c}\right) \left(1-\frac{a}{b c}\right)}
\\
\times
 \frac{ (a;p^2)_k
\left(b ;P^2 \right)_k  } {
\left(\frac{PQR}{p};\frac{PQR}{p}\right)_k \left(\frac{a p P Q}{c
R};\frac{p P Q}{R}\right)_k}
 \frac{ (c;R^2)_k
\left(\frac{a }{b c};Q^2 \right)_k } { \left(\frac{a
pQR}{bP};\frac{pQR}{P}\right)_k \left(\frac{b c p P R}{Q};\frac{ p P
R}{Q}\right)_k}\, R^{2k}\\
= \frac{ (ap^2;p^2)_n \left(b P^2;P^2 \right)_n (cR^2;R^2)_n
\left(\frac{a Q^2}{b c};Q^2 \right)_n  } {
\left(\frac{PQR}{p};\frac{PQR}{p}\right)_n \left(\frac{a p P Q}{c
R};\frac{p P Q}{R}\right)_n\left(\frac{a
pQR}{bP};\frac{pQR}{P}\right)_n \left(\frac{b c p P R}{Q};\frac{ p P
R}{Q}\right)_n},
\end{multline}}
and then in \eqref{6simsum} let $\alpha_i$ be the $i$-th term in the
sum above, and let $\beta_n$ be the quantity on the right side
above.

The identity at \eqref{poly2q} follows upon setting
$P=Q=p=R=q^{m/2}$ and simplifying.
\end{proof}

\section{A Connection with the Prouhet - Tarry - Escott Problem}
We begin with a simple example.

\begin{corollary}\label{6c5}
\begin{multline}\label{6phi52}
_{6} \phi _{5} \left [
\begin{matrix}
q\sqrt{xyz},-q\sqrt{xyz}, y,z, aq,bq\\
\sqrt{xyz},-\sqrt{xyz}, qxy,qxz,a b q
\end{matrix}
; q,x \right ]\\
= \frac{(1-xy)(1-xz)}{ (1-x)(1-xyz)} \;_{4}\phi_{3} \left
[\begin{matrix}
y, z, a, b\\
xy,xz, a b q
\end{matrix}
; q,x q \right ].
\end{multline}
\end{corollary}

\begin{proof}
This time, in Corollary \ref{6c1}, define $\alpha_0 =1$, and for
$n>0$,
\begin{equation*}
 \alpha_n =
 \frac{(aq,bq;q)_n}{(a b q,q;q)_n}-\frac{(aq,bq;q)_{n-1}}{(a b q,q;q)_{n-1}}
 =\frac{(a,b;q)_n  q^{n}}{(a bq,q;q)_n}.
 \end{equation*}
 The  result follows as above.
\end{proof}

The telescoping approach used in Corollary  \ref{6c5} can be
generalized in one direction. We have the following result.

\begin{proposition}\label{6ppte}
Let $x$, $y$ and $q$ be complex numbers with $|x|$, $|q|<1$ and let
$Z$ be an indeterminate. Suppose $a_1, a_2, \dots, a_m$  and  $b_1,
b_2, \dots, b_{m-1}$ are non-zero complex numbers  satisfying
\begin{equation}\label{6abmeq}
\prod_{i=1}^{m}(1-a_i)
=\prod_{i=1}^{m}(Z-a_i)-(Z-1)\prod_{i=1}^{m-1}(Z-b_i).
\end{equation}
Suppose further that $b_i \not =0$, for $1 \leq i \leq m-1$. Then
\begin{multline}
_{m+4} \phi _{m+3} \left [
\begin{matrix}
q\sqrt{xyz},\, -q \sqrt{xyz},\, y,\, z,\, a_1q,\, \dots,\,a_{m-1}q,\, a_m q\\
\sqrt{xyz},\, -\sqrt{xyz},\, q x y,\, qxz,\, b_1q,\,
\dots,\,b_{m-1}q
\end{matrix}
; q,\,x \right ]\\ = \frac{(1-xy)(1-xz)}{ (1-x)(1-xyz)} \,  _{m+2}
\phi _{m+1} \left [
\begin{matrix}
 y,\, z,\, a_1,\, \dots,\,a_{m-1},\, a_m \\
xy,\, xz ,\,  b_1q,\, \dots,\,b_{m-1}q
\end{matrix}
; q,\,x q^m \right ].
\end{multline}
\end{proposition}
\begin{proof}
Define $\alpha_0=1$, and for $n\geq 1$, set
\[
\alpha_n = \frac{(a_1q, a_2 q, \dots,a_{m-1}q, a_m q;q)_{n}} {(b_1q,
b_2 q, \dots,b_{m-1}q, \,q;q)_{n}}-\frac{(a_1q, a_2 q,
\dots,a_{m-1}q, a_m q;q)_{n-1}}{(b_1q, b_2 q, \dots,b_{m-1}q,\,
q;q)_{n-1}}.
\]
By \eqref{6abmeq} with $Z=q^{-n}$,
\[
\alpha_n = \frac{(a_1, a_2 , \dots,a_{m-1}, a_m ;q)_{n}} {(b_1q, b_2
q, \dots,b_{m-1}q, \,q;q)_{n}}\, q^{m n}
\]
and clearly
\begin{equation}\label{absumeq}
\beta_n = \sum_{r=0}^n\alpha_r = \frac{(a_1q, a_2 q, \dots,a_{m-1}q,
a_m q;q)_{n}} {(b_1q, b_2 q, \dots,b_{m-1}q, \,q;q)_{n}}.
\end{equation}
The result follows from Corollary \ref{6c1}.
\end{proof}
It is not clear how to find explicit sets of complex numbers $a_1,
a_2, \dots, a_m$, $b_1, b_2, \dots, b_{m-1}$ satisfying
\eqref{6abmeq}, but  a related problem in number theory provides
solutions for $m \leq 10$ and $m=12$.

The \emph{Prouhet-Tarry-Escott} problem asks for two distinct
 multisets of integers $A = \{a_1, . . . , a_m\}$
and $B = \{b_1, . . . , b_m\}$ such that
\begin{equation}\label{pteeq}
\sum_{i=1}^{m}a_i^e=\sum_{i=1}^{m}b_i^e, \text{ for } e = 1, 2,
\dots , k,
\end{equation}
for some integer $k < m$. If $k = m-1$, such a solution is called
\emph{ideal}. We write
\begin{equation}\label{pteeq2}
\{a_1, . . . , a_m\}\stackrel{k}{=}\{b_1, . . . , b_m\}
\end{equation}
to denote a solution to the Prouhet-Tarry-Escott problem.

 The connection between the Prouhet-Tarry-Escott
problem and the problem mentioned above is contained in the
following proposition (see \cite{BLP03}, page 2065).
\begin{proposition}
The multisets  $A = \{a_1, . . . , a_m\}$ and $B = \{b_1, . . . ,
b_m\}$ form an ideal solution to the Prouhet-Tarry-Escott problem if
and only if
\[
\prod_{i=1}^{m}(Z-a_i)-\prod_{i=1}^{m}(Z-b_i)=C,
\]
for some constant $C$, where $Z$ is an indeterminate.
\end{proposition}
Note that the fact that $b_m=1$ is not a problem, since if
\[
\{a_1, . . . , a_m\} \stackrel{m-1}{=}\{b_1, . . . , b_m\},
\]
then
\[
 \{M a_1+K, . . . , M a_m+K\}\stackrel{m-1}{=} \{M b_1+K, . . .
,M b_m+K\}, \] for constants $M$ and $K$ (see Lemma 1 in \cite{C37},
for example). Note also that if $b_m=1$, then setting $Z=1$ gives
$C=\prod_{i=1}^{m}(1-a_i)$.

Parametric ideal solutions are known for $m=1,\dots , 8$ and
particular numerical solutions are known for $m=9, 10$ and 12.
Although every ideal solution to the Prouhet-Tarry-Escott problem
gives rise to a transformation between basic hypergeometric series,
we will consider just one example. Note also that it is not
necessary, for our purposes, that the $a_i$'s and $b_i$'s be
integers. As above, we assume $x$, $y$ and $q$ are complex numbers,
with $|x|, |q|<1$.
\begin{corollary}\label{6cpte3}
Let $m$ and $n$  be non-zero complex numbers. Set
\begin{align}\label{ab65eq1}
&a_1=-3 m^2+7 n m-2 n^2+1,&&b_1=-3 m^2+8 n m+n^2+1,  \\
&a_2=-2 m^2+8 n m+2 n^2+1,&&b_2=-2 m^2+3 n m-3 n^2+1,\notag\\
&a_3=-m^2-n^2+1,&&b_3=-m^2+10 nm-n^2+1,\notag\\
&a_4=2 m^2+3 n m+n^2+1,&&b_4=2 m^2+2 n m-2 n^2+1,\notag\\
&a_5=m^2+2 n m-3 n^2+1,&&b_5=m^2+7 n m+2 n^2+1,\notag\\
&a_6=10 m n+1.&&\phantom{as}\notag
\end{align}
Then
\begin{multline}
_{10} \phi _{9} \left [
\begin{matrix}
q\sqrt{xyz},\, -q\sqrt{xyz},\, y,\, z,\, a_1q,\, a_2 q, \,a_{3}q,\, a_4 q,\,a_5 q,\,a_6 q\\
\sqrt{xyz},\, -\sqrt{xyz},\, qxy,\, qxz,\, b_1q,\, b_2
q,\,b_{3}q,\,b_4 q,\, b_5 q
\end{matrix}
; q,\,x \right ]\\ = \frac{(1-xy)(1-xz)}{ (1-x)(1-xyz)} \,  _{8}
\phi _{7} \left [
\begin{matrix}
 y,\, z,\, a_1,\, a_2,\, a_{3},\, a_4,\,a_5,\,a_6 \\
xy,\, xz,\, b_1q,\, b_2 q,\, b_{3}q,\,b_4 q,\,b_5 q
\end{matrix}
; q,\,x q^6 \right ].
\end{multline}
\end{corollary}

\begin{proof}
We have from page 629--30 and Lemma 1 in \cite{C37}, that if
{\allowdisplaybreaks
\begin{align}\label{ab65eq}
&a_1=-5 m^2+4 n m-3 n^2+K,&&b_1=-5 m^2+6 n m+3 n^2+K,\\
&a_2=-3   m^2+6 n m+5 n^2+K,&&b_2=-3 m^2-4 n m-5 n^2+K,\notag \\
&a_3=-m^2-10 n   m-n^2+K,&&b_3=-m^2+10 n m-n^2+K,\notag\\
&a_4=5 m^2-4 n m+3 n^2+K,&&b_4=5  m^2-6 n m-3 n^2+K,\notag\\
&a_5=3 m^2-6 n m-5 n^2+K,&&b_5=3m^2+4 n m+5 n^2+K,\notag\\
&a_6=m^2+10 n m+n^2+K,&&b_6=m^2-10 nm+n^2+K,\notag
\end{align}
} then
\[
\{a_1,a_2,a_3,a_4,a_5,a_6\}\stackrel{5}{=}\{b_1,b_2,b_3,b_4,b_5,b_6\}.
\]
We set $b_6=1$, solve for $K$ and back-substitute in \eqref{ab65eq}.
We then replace $m$ by $m/\sqrt{2}$ and $n$ by $n/\sqrt{2}$. This
leads to the values for the $a_i$'s and $b_i$'s given at
\eqref{ab65eq1} and the result follows, as before, from Proposition
\ref{6ppte}.
\end{proof}
We also note each ideal solution to the Prouhet-Tarry-Escott problem
leads to an infinite summation formula, upon letting $n \to \infty$
in \eqref{absumeq}. We give one example.

\begin{corollary}\label{6cpte5}
Let $m$   be a non-zero complex number. Set {\allowdisplaybreaks
\begin{align*}
\{a_i\}_{i=1}^{12}&=\{1 + 170 m, 1 + 126 m, 1 + 209 m, 1 + 87 m, 1 +
234 m, 1 + 62 m,\\ &\phantom{asdf}1 + 275 m,
  1 + 21 m, 1 + 288 m, 1 + 8 m, 1 + 299 m, 1 - 3 m\},\\
\{b_i\}_{i=1}^{11}&=\{1 + 183 m, 1 + 113 m, 1 + 195 m, 1 + 101 m, 1
+ 242 m, 1 + 54 m,\\ &\phantom{asdf}1 +
  269 m, 1 + 27 m, 1 + 294 m, 1 + 2 m, 1 + 296 m
\}.
\end{align*}
} Then {\allowdisplaybreaks
\begin{multline}\label{14phi13}
\,_{12} \phi _{11} \left [
\begin{matrix}
  a_1,\, a_2,\, a_{3},\, a_4,\,a_5,\,a_6,\,a_7,\,a_8,a_9,a_{10},a_{11},a_{12} \\
b_1q,\, b_2 q,\, b_{3}q,\,b_4 q,\,b_5 q,\,b_6 q,b_7
q,b_9q,b_{10}q,b_{11}q
\end{matrix}
; q,\, q^{12} \right ].\\
=\frac{(a_1q,a_2q,a_{3}q,a_4q,a_5q,a_6q,a_7q,a_8q,a_9q,a_{10}q,a_{11}q,a_{12}q;q)_{\infty}}
{(b_1q,\, b_2 q,\, b_{3}q,\,b_4 q,\,b_5 q,\,b_6 q,b_7
q,b_9q,b_{10}q,b_{11}q,q;q)_{\infty}}
\end{multline}
}
\end{corollary}

\begin{proof}
We use a result of Nuutti Kuosa, Jean-Charles Meyrignac and Chen
Shuwen (see \cite{S99}), namely, that if
\begin{align}\label{Kmeq}
A=\{&K + 22 m, K - 22 m, K + 61 m, K - 61 m, K + 86 m, K - 86 m,\\&
K + 127 m,
  K - 127 m, K + 140 m, K - 140 m, K + 151 m, K - 151 m\},\notag\\
B=\{&K + 35 m, K - 35 m, K + 47 m, K - 47 m, K + 94 m, K - 94 m,\notag\\
&K + 121 m,
  K - 121 m, K + 146 m, K - 146 m, K + 148 m, K - 148 m\},\notag
\end{align}
then
\[
A\stackrel{11}{=}B.
\]
\end{proof}

Remark: Note that while the $K$ and $m$ are irrelevant in
\eqref{Kmeq} in so far as finding integer solutions to the
Prouhet-Tarry-Escott problem (since the solution derived another
solution by scaling by $m$ and translating by $K$ is trivially
equivalent to the original solution), solving $B_{12}=1$ for $K$
leaves $m$ as a non-trivial free parameter in \eqref{14phi13}.

\section{Identities of the Rogers-Ramanujan-Slater Type}\label{RRS
section}

We next prove a number of  identities of the Rogers-Ramanujan-Slater
type. We believe these to be new. We first
prove two general transformations.

\begin{corollary}\label{6c1ab}
For $q$ and $x$  inside the unit disc, and integers $a>0$ and $b$,
\begin{multline}\label{6aln=1b}
\sum_{n=0}^{\infty} \frac{(q\sqrt{xyz},-q\sqrt{xyz}, y,
z;q)_{n}x^{n} q^{(a n^2+b n)/2}}
{(\sqrt{xyz},-\sqrt{xyz},qxy,qxz;q)_{n}(-q^{(a+b)/2};q^a)_n} \\
= \frac{\left(1-xy\right)\left(1-xz\right)}
{(1-x)\left(1-xyz\right)}\left( 1-q^{(a-b)/2} \sum_{n=1}^{\infty}
\frac{(y,z;q)_{n}x^{n}q^{(a n^2+(b-2a) n)/2}}{(xy,xz;q)_n
(-q^{(a+b)/2};q^a)_n} \right ).
\end{multline}
\end{corollary}
\begin{proof}
In Corollary \ref{6c1} set $\alpha_0=1$ and, for $n>0$,
\[
\alpha_{n}=\frac{q^{(a n^2+b
n)/2}}{(-q^{(a+b)/2};q^a)_n}-\frac{q^{(a (n-1)^2+b
(n-1))/2}}{(-q^{(a+b)/2};q^a)_{n-1}} =-q^{(a-b)/2}\frac{q^{(a
n^2+(b-2a) n)/2}}{(-q^{(a+b)/2};q^a)_n}.
\]
\end{proof}

Remark: We initially derived \eqref{6aln=1b} as stated (with the aim of using it to derive new identities of Rogers-Ramanujan-Slater type), but it was subsequently pointed out to us that the same argument leads to a more general bi-basic identity, if $q^{a/2}$ is replaced with any complex number $p$ with $|p|<1$, and $q^{b/2}$ is replaced with any non-zero complex number $B$. This gives
\begin{multline}\label{6aln=1bpB}
\sum_{n=0}^{\infty} \frac{(q\sqrt{xyz},-q\sqrt{xyz}, y,
z;q)_{n}(Bx)^{n} p^{n^2}}
{(\sqrt{xyz},-\sqrt{xyz},qxy,qxz;q)_{n}(-Bp;p^2)_n} \\
= \frac{\left(1-xy\right)\left(1-xz\right)}
{(1-x)\left(1-xyz\right)}\left( 1-\frac{p}{B} \sum_{n=1}^{\infty}
\frac{(y,z;q)_{n}(Bx)^{n}p^{n^2-2n}}{(xy,xz;q)_n
(-Bp;p^2)_n} \right ).
\end{multline}

\begin{corollary}\label{6c1ac}
For $q$ and $x$  inside the unit disc, and integers $a>0$ and $b$,
\begin{multline}\label{6aln=1c}
\sum_{n=0}^{\infty} \frac{(q\sqrt{xyz},-q\sqrt{xyz}, y,
z;q)_{n}x^{n} q^{a n^2+b n}}
{(\sqrt{xyz},-\sqrt{xyz},qxy,qxz;q)_{n}}  =
\frac{\left(1-xy\right)\left(1-xz\right)}
{(1-x)\left(1-xyz\right)}\\
\times \left( 1-q^{(a-b)} \sum_{n=1}^{\infty}
\frac{(y,z;q)_{n}x^{n}q^{a n^2+(b-2a)n}(1-q^{2 a n
+b-a})}{(xy,xz;q)_n } \right ).
\end{multline}
\end{corollary}
\begin{proof}
In Corollary \ref{6c1} set $\alpha_0=1$ and, for $n>0$,
\[
\alpha_{n}=q^{a n^2+b n}-q^{a (n-1)^2+b (n-1)} =-q^{a n^2+(b-2a)
n+a-b}(1-q^{2 a n+b-a}).
\]
\end{proof}

Remark: As above, we initially derived \eqref{6aln=1b2} as stated with the aim of using it to derive new identities of Rogers-Ramanujan-Slater type (see below), but  a more general bi-basic identity holds, if $q^{a}$ is replaced with any complex number $p$ with $|p|<1$, and $q^{b}$ is replaced with any non-zero complex number $B$. This gives
\begin{multline}\label{6aln=1bpB2}
\sum_{n=0}^{\infty} \frac{(q\sqrt{xyz},-q\sqrt{xyz}, y,
z;q)_{n}(Bx)^{n} p^{n^2}}
{(\sqrt{xyz},-\sqrt{xyz},qxy,qxz;q)_{n}} \\
= \frac{\left(1-xy\right)\left(1-xz\right)}
{(1-x)\left(1-xyz\right)}\left( 1-\frac{p}{B} \sum_{n=1}^{\infty}
\frac{(y,z;q)_{n}(Bx)^{n}p^{n^2-2n}(1-Bp^{2n-1})}{(xy,xz;q)_n
} \right ).
\end{multline}

\begin{corollary}
\begin{equation}\label{RRS3}
\sum_{n=0}^{\infty} \frac{(1+ q^{-2n+3})q^{n^2+6n}}{(q^4;q^4)_{n}}
=\frac{1}{(q^2,q^3;q^5)_{\infty}(-q^2;q^2)_{\infty}}.
\end{equation}
\begin{equation}\label{RRS3N}
\sum_{n=0}^{\infty} \frac{(1+ q^{-2n+1})q^{n^2+4n}}{(q^4;q^4)_{n}}
=\frac{1}{(q,q^4;q^5)_{\infty}(-q^2;q^2)_{\infty}}.
\end{equation}
\end{corollary}

\begin{proof}
In \eqref{6aln=1c}, set $z=0$, replace $x$ by $x/y$ and let $y \to
\infty$ to get
\begin{multline}\label{RRS3eq1}
\sum_{n=0}^{\infty} \frac{(-x)^n q^{an^2+bn+n(n-1)/2}}{(xq;q)_n} =
(1-x) \\ \times \left( 1-q^{(a-b)} \sum_{n=1}^{\infty}
\frac{(-x)^{n}q^{a n^2+(b-2a)n+n(n-1)/2}(1-q^{2 a n
+b-a})}{(x;q)_n } \right ).
\end{multline}
Next, let $a=-1/4$, $b=1$, replace $q$ by $q^4$  and let $x\to 1$
to get
\[
\sum_{n=0}^{\infty} \frac{(-1)^n q^{n^2+2n}}{(q^4;q^4)_n} =
-q^{-5} \sum_{n=1}^{\infty} \frac{(-1)^{n}q^{ n^2+4n}(1-q^{-2 n
+5})}{(q^4;q^4)_{n-1} }.
\]
Replace $q$ by $-q$, re-index the right side by replacing $n$ by
$n+1$ and \eqref{RRS3} follows from the following identity of
Rogers (\cite{R94}, page 331):
\[
\sum_{n=0}^{\infty} \frac{
q^{n^2+2n}}{(q^4;q^4)_n}=\frac{1}{(q^2,q^3;q^5)_{\infty}(-q^2;q^2)_{\infty}}.
\]
The identity at \eqref{RRS3N} follows similarly, using instead
$a=-1/4$, $b=1/2$ in \eqref{RRS3eq1} and employing another
identity of Rogers (\cite{R94}, page 330):
\[
\sum_{n=0}^{\infty} \frac{
q^{n^2}}{(q^4;q^4)_n}=\frac{1}{(q,q^4;q^5)_{\infty}(-q^2;q^2)_{\infty}}.
\]
\end{proof}

\begin{corollary}
\begin{equation}\label{RRS6}
\sum_{n=0}^{\infty} \frac{(b,q^3/b;q)_n
q^{n(n+1)/2}}{(q^2;q^2)_{n+1}(q;q)_{n}}
=\frac{(q^4/b,bq;q^2)_{\infty}}{(q;q)_{\infty}}.
\end{equation}

\begin{equation}\label{RRS6eq2}
\sum_{n=0}^{\infty} \frac{(1-q^{2n+1})(-q^3;q^2)_n
q^{n^2}}{(q^2;q^2)_{n}}
=\frac{1}{(q^3,q^4,q^5;q^8)_{\infty}}.
\end{equation}

\begin{equation}\label{RRS6eq3}
\sum_{n=0}^{\infty}
\frac{(1-q^{2n-1})(-q^3;q^2)_nq^{n^2-2n}}{(q^2;q^2)_{n}}
=\frac{1}{(q,q^4,q^7;q^8)_{\infty}}.
\end{equation}

\begin{equation}\label{RRS6eq4}
1+\sum_{n=1}^{\infty} \frac{(-q;q)_n q^{(n^2-n)/2}}{(q;q)_{n-1}} =
\frac{(-1;q)_{\infty}(-q^6,-q^{10},q^{16};q^{16})_{\infty}}{(q^4;q^4)_{\infty}}.
\end{equation}

\begin{equation}\label{RRS6eq5}
-1+\sum_{n=1}^{\infty} \frac{(-q;q)_n q^{(n^2-n)/2}}{(q;q)_{n-1}}
=
q\frac{(-1;q)_{\infty}(-q^2,-q^{14},q^{16};q^{16})_{\infty}}{(q^4;q^4)_{\infty}}.
\end{equation}

\end{corollary}

\begin{proof}
In \eqref{6phi52}, let $z=0$, replace $x$ by $x/y$ and let $y \to
\infty$ to get
\begin{equation*} \sum_{n=0}^{\infty}
\frac{(aq,bq;q)_n (-x)^n q^{n(n-1)/2}}{(qx,abq,q;q)_{n}}=(1-x)
\sum_{n=0}^{\infty} \frac{(a,b;q)_n (-xq)^n
q^{n(n-1)/2}}{(x,abq,q;q)_{n}}.
\end{equation*} Then set $x=-q$, $a=b/q$ and then use Andrews' $q$-Bailey
identity,
 \begin{equation*}
\sum_{n=0}^{\infty} \frac{(b,q/b;q)_n c^n
q^{n(n-1)/2}}{(c;q)_{n}(q^2;q^2)_{n}}
=\frac{(cq/b,bc;q^2)_{\infty}}{(c;q)_{\infty}}
 \end{equation*}
with $c=q^2$, to sum the right side. Finally, replace $b$ by $b/q$
and \eqref{RRS6} follows after a slight manipulation.

For the remaining identities, in \eqref{6phi52} replace $x$ by
$x/y$, let $y \to \infty$ and then set $z=x$ and $b=0$ to get
\begin{equation*}
\sum_{n=0}^{\infty} \frac{(1+x q^{n})(aq;q)_n (-x)^n
q^{n(n-1)/2}}{(q;q)_{n}}=\sum_{n=0}^{\infty} \frac{(a;q)_n (-x)^n
q^{n(n+1)/2}}{(q;q)_{n}}.
\end{equation*}
For \eqref{RRS6eq2} and \eqref{RRS6eq3}, replace $q$ by $q^2$, set
$a=-q$ and, respectively, $x=-q$ and $x=-1/q$, and use the
 G\"{o}llnitz-Gordon-Slater identities (\cite{G65}, \cite{G67}, \cite{S52})
\begin{align}\label{GG1}
\sum_{n=0}^{\infty} \frac{ q^{n^2+2n}(-q;q^2)_{n}}{(q^2;q^2)_{n}}
&= \frac{ 1}{(q^3;q^{8})_{\infty}(q^4;q^{8})_{\infty}
(q^{5};q^{8})_{\infty}},\\
\sum_{n=0}^{\infty} \frac{ q^{n^2}(-q;q^2)_{n}}{(q^2;q^2)_{n}} &=
\frac{ 1}{(q;q^{8})_{\infty}(q^4;q^{8})_{\infty}
(q^{7};q^{8})_{\infty}}, \notag
\end{align}
to sum the respective right sides.

For \eqref{RRS6eq4}, set $a=x=-1$ and use the following identity
of Gessel and Stanton (\cite{GS83}, page 196)
\begin{equation*}
1+\sum_{n=1}^{\infty} \frac{(-q;q)_{n-1} q^{(n^2+n)/2}}{(q;q)_{n}}
=
\frac{(-q;q)_{\infty}(-q^6,-q^{10},q^{16};q^{16})_{\infty}}{(q^4;q^4)_{\infty}}.
\end{equation*}
to sum the resulting right side. The identity at \eqref{RRS6eq5}
follows similarly, again with $a=x=-1$, upon using another
identity of Gessel and Stanton (\cite{GS83}, page 196)
\begin{equation*}
\sum_{n=0}^{\infty} \frac{(-q;q)_{n} q^{(n^2+3n)/2}}{(q;q)_{n+1}}
=
\frac{(-q;q)_{\infty}(-q^2,-q^{14},q^{16};q^{16})_{\infty}}{(q^4;q^4)_{\infty}}.
\end{equation*}
\end{proof}

\begin{corollary}
\begin{equation}\label{S69eq}
\sum_{n=0}^{\infty}\frac{(-q^2;q^2)_{n+1}q^{n^2+2n}}{(q;q)_{2n+3}}=
2(-q^2,-q^{14},q^{16};q^{16})_{\infty}\frac{(-q;q^{2})_{\infty}}{(q^2;q^{2})_{\infty}}-\frac{1}{1-q}.
\end{equation}
\begin{equation}\label{S121eq}
\sum_{n=0}^{\infty}\frac{(-q^2;q^2)_{n}q^{n^2}}{(q;q)_{2n+1}}=
2\frac{(q^2,q^{14},q^{16};q^{16})_{\infty}(q^{12},q^{20};q^{32})_{\infty}}{(q;q)_{\infty}}-1.
\end{equation}
\end{corollary}
\begin{proof}
We use \eqref{6aln=1b} to prove these identities. First, let $z\to
0$ and replace $q$ with $q^2$ to get
\begin{multline}\label{6aln=1b2}
\sum_{n=0}^{\infty} \frac{(y;q^2)_{n}x^{n} q^{a n^2+b n}}
{(q^2xy;q^2)_{n}(-q^{a+b};q^{2a})_n} \\
= \frac{\left(1-xy\right)} {(1-x)}\left( 1-q^{a-b}
\sum_{n=1}^{\infty} \frac{(y;q^2)_{n}x^{n}q^{a n^2+(b-2a) n}}{(x
y;q^2)_n (-q^{a+b};q^{2a})_n} \right ).
\end{multline}
 For \eqref{S69eq}, set $a=1$,
$b=2$, $y=-q^2$, and  $x=-1$. Replace $q$ with $-q$, divide both
sides by $(1-q)(1-q^2)$ and use Slater's identity \textbf{69} to sum
the resulting left side:
\begin{equation*}
\sum_{n=0}^{\infty}\frac{(-q^2;q^2)_{n}q^{n^2+2n}}{(q;q)_{2n+2}}=
(-q^2,-q^{14},q^{16};q^{16})_{\infty}\frac{(-q;q^{2})_{\infty}}{(q^2;q^{2})_{\infty}}.
\end{equation*}
The result follows after some slight manipulation.

The proof of \eqref{S121eq} is similar, except we set  $a=1$, $b=0$,
$y=-1$, and $x=-1$, replace $q$ with $-q$, and use Slater's identity
\textbf{121}:
\begin{equation*}
1+\sum_{n=1}^{\infty}\frac{(-q^2;q^2)_{n-1}q^{n^2}}{(q;q)_{2n}}=
\frac{(q^2,q^{14},q^{16};q^{16})_{\infty}(q^{12},q^{20};q^{32})_{\infty}}{(q;q)_{\infty}}.
\end{equation*}
\end{proof}

\section{Representation of False Theta Series as basic
Hypergeometric Series}

In this section we derive some new representations of the false
theta series $\sum_{n=0}^{\infty} (-1)^n q^{n(n+1)/2}$ and
$\sum_{n=0}^{\infty}  q^{n(3n+1)/2}(1-q^{2n+1})$, as basic
hypergeometric series.

On page 13 of the Lost Notebook \cite{R88} (see also \cite[page
229]{AB05}), Ramanujan recorded the following identity (amongst
others in a similar vein):
\begin{equation}\label{R1}
\sum_{n=0}^{\infty}\frac{(q;q^2)_{n}(-1)^n
q^{n^2+n}}{(-q;q)_{2n+1}}=\sum_{n=0}^{\infty} (-1)^n q^{n(n+1)/2}.
\end{equation}
On page 37 of the Lost Notebook, he recorded the identities
\begin{align}\label{R2}
\sum_{n=0}^{\infty}  q^{n(3n+1)/2}(1-q^{2n+1})
&=\sum_{n=0}^{\infty}\frac{q^{2n^2+n}}{(-q;q)_{2n+1}}\\
&=\sum_{n=0}^{\infty}\frac{(-1)^nq^{n(n+1)/2}}{(-q;q)_{n}}.\notag
\end{align}

The identity that follows from equating the left side to the second
right side above also follows as a special case of a more general
identity first stated by Rogers \cite{R17}.

We use these identities in conjunction with \eqref{6aln=1b2} to
prove the following.
\begin{corollary}
\begin{equation}\label{ft1}
1-\sum_{n=0}^{\infty}\frac{(q;q^2)_{n}(-1)^n
q^{n^2-n}}{(-1;q)_{2n+1}}=\sum_{n=0}^{\infty} (-1)^n q^{n(n+1)/2}.
\end{equation}
\begin{equation}\label{ft2}
\frac{2}{1+q}-\sum_{n=0}^{\infty}\frac{q^{2n^2+3n}}{(-q;q)_{2n+1}(1+q^{2n+3})}
=\sum_{n=0}^{\infty} q^{n(3n+1)/2}(1-q^{2n+1}).
\end{equation}
\begin{equation}\label{ft3}
\frac{1}{2}+\sum_{n=0}^{\infty}\frac{(-1)^nq^{n(n+1)/2}}{(-1;q)_{n+2}}=\sum_{n=0}^{\infty}
q^{n(3n+1)/2}(1-q^{2n+1}).
\end{equation}
\end{corollary}
\begin{proof}
For \eqref{ft1}, set $a=b=1$, $y=q$ and $x=-1$ in \eqref{6aln=1b2}.
Then divide both sides of the resulting identity by $1+q$, so that
the left side becomes the left side of \eqref{R1}. The result
follows after re-indexing the resulting sum on the right side,
together with a little manipulation.

For \eqref{ft2}, replace $x$ with $x/y$ in \eqref{6aln=1b2} and let
$y\to \infty$ to get
\begin{multline}\label{6aln=1b3}
\sum_{n=0}^{\infty} \frac{(-x)^{n} q^{(a+1) n^2+(b-1) n}}
{(q^2x;q^2)_{n}(-q^{a+b};q^{2a})_n} \\
= (1-x)\left( 1-q^{a-b} \sum_{n=1}^{\infty} \frac{(-x)^{n} q^{(a+1)
n^2+(b-2a-1) n}}{(x ;q^2)_n (-q^{a+b};q^{2a})_n} \right ).
\end{multline}
Then set $a=1$, $b=2$, $x=-1$, and divide both sides by $1+q$ so
that the left side becomes the first right side of \eqref{R2}. The
result again follows, upon re-indexing the sum on the right side.

To get \eqref{ft1}, set $y=0$ in \eqref{6aln=1b2}, then $a=b=1/2$
and $x=-1$, so the left side becomes the second right side in
\eqref{R2}. The result likewise follows after re-indexing the
resulting sum on the right side.
\end{proof}


{\bf Received: Month xx, 200x}

\end{document}